\documentclass[12pt]{amsart}
\usepackage{amsfonts}
\usepackage{amssymb,latexsym}
\usepackage{enumerate}
\usepackage{mathrsfs}
\usepackage{amsmath}
\usepackage{amsthm}

scaled\magstephalf

\newtheorem{Thm}{Theorem}[section]
\newtheorem{Cor}[Thm]{Corollary}
\newtheorem{Lem}[Thm]{Lemma}
\newtheorem{Rem}[Thm]{Remark}

\numberwithin{equation}{section}

\begin{document}
\setlength{\baselineskip}{1.2\baselineskip}
\title[Neumann Problem]{Mean Curvature Flows of Graphs with Neumann Boundary condition  }

\author{Jinju Xu}
\address {Department of Mathematics\\
        Shanghai University \\
         Shanghai 200444 CHINA }
         \email{jjxujane@shu.edu.cn}


\maketitle

\begin{abstract}
In this paper, we study the mean curvature flow of graphs with Neumann boundary condition. The main aim is to use the maximum principle
 to get the boundary gradient estimate for solutions. In particular,
 we obtain the corresponding existence theorem for the mean curvature flow of graphs.
\end{abstract}
\section{Introduction}

In this note, we  consider the following   Neumann boundary value problem for  the  general mean curvature flow of graphs
\begin{align}
 u_t-\sum_{1\leq i,j\leq n}a^{ij}(Du)u_{ij} =&-f(x,u,Du)   \quad\text{in}\quad \Omega\times[0,\infty), \label{1.1}\\
 \frac{\partial u}{\partial \gamma} = &\psi(x, u)  \quad\text{on} \quad\partial\Omega\times[0,\infty),\quad u|_{t=0}=u_0\quad\text{in}\quad \Omega,\label{1.2}
\end{align}
 where $\Omega$ is a bounded domain in $\mathbb R^n $ with $C^3$ boundary $\partial \Omega $
 and  $\gamma$ is an inward unit normal vector to $\partial\Omega $; $Du=(\frac{\partial u}{\partial x_1},\cdots,\frac{\partial u}{\partial x_n})$ denotes the space gradient of $u$;
 $f: \overline\Omega\times\mathbf R\times\mathbf R^{n}\rightarrow \mathbf R$ and  $ \psi: \overline\Omega\times\mathbf R\rightarrow \mathbf R$
are given  functions,
and $u_0:\overline\Omega\rightarrow \mathbf R$, the initial value, is a smooth function and satisfies
\begin{align}
\frac{\partial u_0}{\partial \gamma} -\psi(x, u_0)=0\quad\text{on} \quad\partial\Omega.\label{1.3}
\end{align}
Here
$$a^{ij}(p)=\delta_{ij}-\frac{p_ip_j}{1+|p|^2},\quad \text{for}\quad p\in\mathbf R^n.$$

This problem \eqref{1.1}-\eqref{1.3} describes the evolution of graph $(u(x,t))$ by its mean curvature
in the direction of the unit normal vector  with Neumann boundary value condition.

This problem  has been  studied by many authors. For $f=0$ and $\psi=0$ in \eqref{1.1}-\eqref{1.3}, G.Huisken \cite{Huis89}  proved that  the solution remains smooth and bounded and asymptotically converges to a constant function.
 B.Andrews and  J.Clutterbuck \cite{AC} studied the problem for a convex domain  and  initial data $u_0\in C(\overline \Omega)$.
 For the general prescribed  contact angle boundary value condition
 \begin{align}
 \frac{\partial u}{\partial \gamma} = &\phi(x)\sqrt{1+|Du|^2}, \quad\text{on} \quad\partial\Omega\times[0,\infty)\label{1.4}
 \end{align}
 with the initial data $u_0$ satisfying
 \begin{align}
 \frac{\partial u_0}{\partial \gamma} = &\phi(x)\sqrt{1+|Du_0|^2} \quad\text{on} \quad\partial\Omega, \label{1.5}
 \end{align}
where  the prescribed  contact angle given by $\cos^{-1}\phi$  and $\phi\in C^{\infty}(\partial \Omega)$ with $|\phi|\le \phi_0<1$, $\phi_0$ is a constant.
  Altschuler and Wu \cite{AW94} proved that for $n=2$, if $\Omega$ is strictly convex
 and $|D\phi|<\min \kappa(\partial\Omega)$, where $\kappa$ is the curvature of $\partial\Omega$, then the solution of \eqref{1.1}($f=0$),\eqref{1.4} with \eqref{1.5}
 either converges to a minimal surface (when $\int_{\partial\Omega}\phi=0$),
 or behaves like moving by a vertical translation as $t$ approaches infinity. For $\forall n\ge 2$
 and for general smooth $\Omega$ and  $\phi$, Guan \cite{GB96} proved
 that there exists a smooth  solution of \eqref{1.1},\eqref{1.4}-\eqref{1.5}. The author can  give a new proof of \eqref{1.1},\eqref{1.4}-\eqref{1.5}
 for general case by the method in \cite{Xu}.

  In this note, we consider the mean curvature flows for general  Neumann problem \eqref{1.1}-\eqref{1.3}. To study the behavior of  graphs,
 it is important to derive the gradient bound for solution of \eqref{1.1}-\eqref{1.3}.
As the interior gradient estimate for \eqref{1.1} has been studied by Ecker and Huisken in \cite{EH89}, a key point is to derive
 the boundary gradient estimate in this note.

Now let's state our main result.
\begin{Thm}\label{Thm1.1}
Assume that $\Omega \subset \mathbb R^n $ is a bounded $C^3$ domain, $n\geq 2$. Let  $\gamma$ be the inward unit normal vector to $\partial\Omega $.
Suppose $u\in C^{2}(\overline\Omega)\bigcap C^{3}(\Omega)\times [0,T]$ for any fixed $T>0$ is a bounded  solution of \eqref{1.1}-\eqref{1.3} with $|u|\leq M_0$,
where $M_0>0$ is a constant.
We assume $f(x,z,p) \in C^{1}(\overline\Omega\times [-M_0, M_0]\times\mathbb R^n)$
and $\psi(x,z) \in C^{3}(\overline\Omega\times [-M_0, M_0])$
satisfying the following conditions
\begin{align}
f_z(x,z,p)\geq &0,
\label{1.7}\\
\frac{|f_{x}|}{|p|}+\sum_{1\leq j\leq n}|f_{p_j}|+|f-\sum_{1\leq j\leq n}f_{p_j}p_j|\leq & o(\log |p|),\quad
\text{as} \,\, |p|\rightarrow\infty,
\label{1.8}
\end{align}
and for some positive constant $L$
\begin{align}
|\psi(x,z)|_{C^3(\overline\Omega\times[-M_0, M_0])}\leq&  L.\label{1.10}
\end{align}
Then there exists a small positive constant
$\mu_0$ such that we have the following boundary gradient estimate for any fixed $T>0$,
$$\sup_{\overline\Omega_{\mu_0}\times [0,T]}|Du|\leq \max\{M_1, M_2\},$$
where $M_1$ is a positive constant depending only on $n, \mu_0, M_0$, which is from the interior gradient estimates;
$M_2$ is  a positive constant depending only on $n, \Omega, \mu_0, M_0, L$, and $d(x) =\texttt{dist}(x, \partial\Omega), \Omega_{\mu_0} = \{x \in \Omega: d(x)<\mu_0\}.$
\end{Thm}

\begin{Rem}\label{Rem1.2}
In particular, for $f=\tilde{f}(x,u)\sqrt{1+|Du|^2}$,  Theorem~\ref{Thm1.1}  holds only if $\tilde{f}_z\ge 0$,
and $|\tilde{f}|+|\tilde{f}_x|$ is bounded.
\end{Rem}

 As we stated before, there is a standard interior gradient estimates for the mean curvature flow.
\begin{Rem}[\cite{EH89}]\label{Rem1.1}
Suppose for any fixed $T>0$, $u\in C^{3}(\Omega)\times[0,T]$ is a bounded solution for the equation \eqref{1.1}  with  $|u|\leq M_0$,
and if $f \in C^{1}(\overline\Omega \times [-M_0, M_0]\times\mathbb R^n)$ satisfies the condition \eqref{1.7}-\eqref{1.8}, then for any subdomain
$\Omega'\subset\subset\Omega$, we have
\begin{align}\sup_{\Omega'\times[0,T]}|Du|\leq M_1, \label{1.11a}\end{align}
where $M_1$ is a positive constant depending only on $n,  M_0, \texttt{dist} (\Omega', \partial\Omega)$.
\end{Rem}

 By the standard theory, we obtain the long time  existence of solutions.
\begin{Thm}\label{Thm1.2}
Assume that $\Omega \subset \mathbb R^n $ is a bounded  $C^{3}$ domain, $n\geq 2$. Let  $\gamma$ be the inward unit normal vector to $\partial\Omega $.
Under the conditions \eqref{1.7}-\eqref{1.10} and $\psi_u\ge 0$,
then for some $\alpha\in (0,1)$ there exists a  unique $ C^{2,\alpha}(\overline\Omega\times[0,\infty))$ solution of  \eqref{1.1}-\eqref{1.3}.
\end{Thm}

The rest of the paper is organized as follows. In section 2, we  will give the definitions and some notations,
 and   derive the following a prior estimates for the solution
\begin{align}
 |u_t|\le C_1,\quad M_T\equiv\max_{\overline\Omega\times[0,T]}|u(x,t)-u_0(x)| \le C_1T.\label{1.11}
 \end{align}
 And in section 3, we give the  boundary gradient estimate
 \begin{align}
 |Du|\le C_2e^{C_3M_T}\quad\text{in}\quad \overline\Omega_{\mu_0}\times[0,T],\label{1.12}
\end{align}
where $\mu_0$ is a fixed small positive number to be determined later and $ C_1, C_2,  C_3$ are uniform constants independent of $T$.
Then we finish the proof of  the main Theorem~\ref{Thm1.1} and obtain the long time existence of solutions. At last, we give an application.

\section{$|u_t|$- estimates}
We denote by $\Omega$ a bounded  domain in $\mathbb{R}^n$, $n\geq 2$,  $\partial \Omega\in C^{3}$,   set
\begin{align*}
 d(x)=\texttt{dist}(x,\partial \Omega),
 \end{align*}
 and
\begin{align*}
 \Omega_\mu=&\{{x\in\Omega:d(x)<\mu}\}.
 \end{align*}
Then it is well known that there exists a positive constant $\mu_{1}>0$ such that $d(x) \in C^3(\overline \Omega_{\mu_{1}})$.
 As in Simon-Spruck \cite{SS76} or Ma-Xu \cite{MX14},  we can take $\gamma= D d$ in $\Omega_{\mu_{1}}$ and note that
   $\gamma$ is a $C^2(\overline \Omega_{\mu_{1}})$ vector field. We also have the following formulas
\begin{align}\label{2.1}
|D\gamma|+|D^2\gamma|\leq& C(n,\Omega) \quad\text{in}\quad \Omega_{\mu_{1}},\\
 \sum_{1\leq i\leq n}\gamma^iD_j\gamma^i=0,  \quad\sum_{1\leq i\leq n}\gamma^iD_i\gamma^j=&0, \quad|\gamma|=1 \quad\text{in} \quad\Omega_{\mu_{1}}.
\end{align}
As in \cite{MX14}, we define
 \begin{align}\label{2.2}
c^{ij}=&\delta_{ij}-\gamma^i\gamma^j  \quad \text{in} \quad \Omega_{\mu_{1}},
\end{align}
 and for a vector $\zeta \in R^n$, we write $\zeta'$ for the vector with $i-$th component $ \sum_{1\leq j\leq n}c^{ij}\zeta_j$. So
 \begin{align}\label{2.3}
|D'u|^2=& \sum_{1\leq i,j\leq n}c^{ij}u_iu_j.
\end{align}
Setting
\begin{align*}
v=(1+|D u|^2)^{\frac{1}{2}},\quad a^{ij}=a^{ij}(D u),
\end{align*}
then\begin{align}
a^{ij}=\delta_{ij}-\frac{u_iu_j}{v^2},\quad \sum_{1\leq i,j\leq n}a^{ij}u_iu_j=1-\frac{1}{v^2}\label{2.4}
\end{align}

Now we first establish  \eqref{1.11} and then \eqref{1.12} in the next section.
\begin{Lem}\label{Lem2.1}
Assume that $f_u\ge 0$ and $\psi_u\ge0$.
Then we obtain the estimate
\begin{align*}
\max_{\overline\Omega\times[0,\infty)}|u_t|=\max_{\overline\Omega}|u_t(\cdot,0)|.
\end{align*}

\end{Lem}
\begin{proof} It suffices to prove the following: For any fixed $T>0$, if $u_t$ admits a positive
local maximum at some point $(x_0,t_0)\in\overline\Omega\times[0,T]$, that is
\begin{align*}
u_t(x_0,t_0)= \max_{\overline\Omega\times[0,T]}u_t\ge 0,
\end{align*}
then $t_0=0$. Now suppose $t_0>0$.
It is easy to calculate that $u_t$ satisfies the equation
\begin{align}
\frac{d}{dt}u_{t}=& \sum_{1\leq i,j\leq n}a^{ij}(u_t)_{ij}
-\frac{1}{2u_t}\sum_{1\leq i,j\leq n}a^{ij}(u_t)_{i}(u_t)_{j}-\sum_{1\leq j\leq n}f_{p_j}(u_t)_{j}\notag\\
&-\frac{2}{v}\sum_{1\leq i,j\leq n}a^{ij}v_i(u_t)_{j}-f_uu_t.\label{2.5}
\end{align}
 Then from $f_u\ge 0$, we obtain that $u_t$ satisfies the parabolic maximum principle. Hence $x_0\in\partial\Omega$.

On the other hand, we differentiate the Neumann boundary condition along the normal vector, and get from $\psi_u\ge0$
\begin{align}\label{2.6}
(u_t)_{\gamma}=& \psi_uu_t\ge0.
\end{align}
But this contradicts the Hopf Lemma at $(x_0,t_0)$.
\end{proof}
By Lemma~\ref{Lem2.1}, it follows that
\begin{align}
\max_{\overline\Omega\times[0,T]}|u(x,t)-u_0(x)| \le C_1T,
\end{align}
where $C_1$ is a constant independent of $T$.
\section{ $|Du|$-estimates}
Now we begin to prove Theorem~\ref{Thm1.1}.  We follow the technique in  Ma-Xu\cite{MX14}.
Setting $w=u-\psi(x,u)d$,  we choose the following auxiliary function
$$\Phi(x,t)=\log|Dw|^2e^{1+M_0+u}e^{\alpha_0 d}, \quad  (x,t) \in \overline \Omega_{\mu_{0}}\times[0,T],
\quad 0<\mu_0<\mu_1$$ where $\alpha_0=|\psi|_{C^0(\overline\Omega\times[-M_0, M_0])}+C_0+1 $, $C_0$ is
a positive constant depending only on $n,\Omega$.

For simplification, we let
\begin{align}\label{3.0}
 \varphi(x,t)= \log \Phi(x,t) =\log\log|Dw|^2+h(u)+g(d),
 \end{align}
 where
\begin{align}\label{3.1}
 h(u)=1+M_0+u,\quad g(d)=\alpha_0 d.
 \end{align}

Suppose that
$\varphi(x,t)$ attains its maximum at $(x_0,t_0)\in \overline \Omega_{\mu_{0}}\times[0,T]$, where $0<\mu_0<\mu_1$
is a sufficiently small number which  shall be decided   later. Now we divide three cases to prove Theorem~\ref{Thm1.1}.

{\bf Case 1.}  $x_0 \in \partial\Omega$,
 then we shall  get the bound of $|Du|(x_0,t_0)$.

{\bf Case 2.}  $x_0 \in\partial\Omega_{\mu_0}\bigcap\Omega$, then we shall get the estimates via the  interior gradient bound in Remark~\ref{Rem1.1}.

{\bf Case 3.}  $x_0\in \Omega_{\mu_0}$, in this case for the sufficiently small constant $\mu_0>0$,  then we can use the maximum principle to get the bound of $|Du|(x_0,t_0)$.

\begin{proof}{\bf Case 1.} $x_0 \in \partial\Omega$.
We differentiate $\varphi$ along the normal direction.
\begin{align}\label{5.2}
\frac{\partial\varphi}{\partial\gamma}=&\frac{\sum_{1\leq i\leq n}(|Dw|^2)_i\gamma^i}{|Dw|^2\log|Dw|^2}+h'u_{\gamma}+g'.
\end{align}
Since
\begin{align}
w_i=&u_i-\psi_u u_i d -\psi_{x_i} d-\psi\gamma^i,\label{5.3}\end{align}
and
\begin{align}
|Dw|^2=&|D'w|^2+w^2_\gamma,\label{5.4}
\end{align}
we have
\begin{align}
w_\gamma=&u_\gamma-\psi_uu_\gamma d-\psi_{x_i}\gamma^i d-\psi=0\quad\text{on}\quad\partial\Omega,\label{5.5}
\end{align}
and
 \begin{align}
(|Dw|^2)_i=&(|D'w|^2)_i\quad\text{on}\quad\partial\Omega.\label{5.6}
\end{align}
Applying  \eqref{2.1}, \eqref{2.3} and \eqref{5.6}, it follows that
\begin{align}\label{5.7}
(|Dw|^2)_i\gamma^i=&(|D'w|^2)_i\gamma^i
=2\sum_{1\leq k,l\leq n}c^{kl}w_{ki}w_l\gamma^i
=2\sum_{1\leq k,l\leq n}c^{kl}u_{ki}u_l\gamma^i-2\sum_{1\leq k,l\leq n}c^{kl}u_lD_k\psi,
\end{align}
where
\begin{align*}
D_k\psi= \psi_{x_k} + \psi_u u_k.
\end{align*}
Differentiating the Neumann boundary value condition in \eqref{1.2} with respect to tangential direction,   we have
\begin{align}\label{5.8}
\sum_{1\leq k\leq n}c^{kl}(u_{\gamma})_k=&\sum_{1\leq k\leq n}c^{kl}D_k\psi.
\end{align}
It follows that
\begin{align}\label{5.9}
\sum_{1\leq k,i\leq n}c^{kl}u_{ik}\gamma^i=&-\sum_{1\leq k,i\leq n}c^{kl}u_i(\gamma^i)_k+\sum_{1\leq k\leq n}c^{kl}D_k\psi.
\end{align}
Inserting \eqref{5.9} into \eqref{5.7} and combining \eqref{1.2}, \eqref{5.2}, we have
\begin{align}\label{5.10}
|Dw|^2\log|Dw|^2\frac{\partial\varphi}{\partial\gamma}(x_0,t_0)
=&(g'(0)+h'\psi)|Dw|^2\log|Dw|^2-2\sum_{1\leq i,k,l\leq n}c^{kl}u_iu_l(\gamma^i)_k.
\end{align}
From \eqref{5.3}, we obtain
\begin{align}\label{5.11}
|Dw|^2=&|Du|^2-\psi^2\quad\text{on}\quad\partial\Omega.
\end{align}
Assume $|Du|(x_0,t_0)\ge \sqrt{100+3|\psi|^2_{C^0(\overline\Omega\times[-M_0, M_0])}}$, otherwise we get the estimates. At $(x_0,t_0)$,   we have
\begin{align}
\frac{1}{2}|Du|^2\leq&|Dw|^2\leq |Du|^2,\quad |Dw|^2\ge 50.\label{5.12}
\end{align}
Inserting  \eqref{5.12}  into \eqref{5.10}, we have
\begin{align}\label{5.14}
\frac{\partial\varphi}{\partial\gamma}(x_0,t_0)
\geq&\alpha_0-|\psi|_{C^0(\overline\Omega\times[-M_0, M_0])}-\frac{\sum_{1\leq i,k,l\leq n}|2c^{kl}u_iu_l(\gamma^i)_k|}{|Dw|^2\log|Dw|^2}\notag\\
\geq&\alpha_0-|\psi|_{C^0(\overline\Omega\times[-M_0, M_0])}-C_0\notag\\
>&0.
\end{align}
On the other hand, we have
$$\frac{\partial\varphi}{\partial\gamma}(x_0,t_0)\leq 0,$$
it is a contradiction to \eqref{5.14}.

Then we have
\begin{align}\label{5.15}
|Du|(x_0,t_0)\leq\sqrt{100+2|\psi|^2_{C^0(\overline\Omega\times[-M_0, M_0])}}.
\end{align}
\end{proof}

{\bf Case 2.} $x_0 \in\partial\Omega_{\mu_0}\bigcap\Omega$.
This is due to interior gradient estimates. From Remark~\ref{Rem1.1}, we have
 \begin{align}\label{5.16}
\sup_{\partial\Omega_{\mu_0}\bigcap\Omega\times[0,T]}|Du|\leq \tilde{M}_1,
\end{align}
where $\tilde{M}_1$ is a positive constant depending only on $n, M_0, \mu_0, L_1$.

{\bf Case 3.} $x_0\in \Omega_{\mu_0}$.
In this case, $x_0$ is a critical point of $\varphi$.
 We choose the normal coordinate at $x_0$, by rotating the coordinate system
suitably, we may assume that $w_i(x_0,t_0)=0,\,2\leq i\leq n$ and $w_1(x_0,t_0)=|Dw|(x_0,t_0)>0$.
 And we can further assume that the matrix $(w_{ij}(x_0,t_0))(2\leq i,j\leq n)$ is diagonal.
 We denote that  the choice of coordinate has been used in Wang \cite{Wang98}  and is slightly different from  Ma-Xu\cite{MX14}.
 The calculation in this paper is more simple than that in \cite{MX14}.

Setting $$\mu_2 \le \frac{1}{10
0L_2}$$ such that
 \begin{align}\label{5.16a}
 |\psi_u| \mu_2 \le \frac{1}{100
}, \quad \text{then}\quad \frac{99}{100}\le 1-\psi_u \mu_2 \le
 \frac{101}{100}.
\end{align}
We can choose $$ \mu_0=\frac{1}{2} \min\{ \mu_1, \mu_2,  1 \}.$$
Since we have let
\begin{align*}
w=&u-G,\quad G=\psi(x,u)d,
\end{align*}
then we have
\begin{align}
w_k=(1-G_u)u_k-G_{x_k}.\label{3wk}
\end{align}
Since at $(x_0,t_0)$,
\begin{align}
w_1=&(1-G_u)u_1-G_{x_1}=(1-G_u)u_1 -\psi_{x_1}d-\psi\gamma^1,\label{3wi2a}\\
0=w_i=&(1-G_u)u_i-G_{x_i}=(1-G_u)u_i-\psi_{x_i}d-\psi\gamma^i,\quad i=2,\ldots,n.\label{3wi2}
\end{align}
So from the above relation,  at $(x_0,t_0)$, we can assume
\begin{align}
u_1 \ge 200(1+ |\psi|_{C^1(
\overline\Omega\times[-M_0, M_0])}),\label{3wk}
\end{align}
then
\begin{align}\label{5.18}
\frac{19}{20}u_1\leq w_1\leq \frac{21}{20}u_1, \quad \frac{91}{100}w_1^2\leq |Du|^2\leq \frac{111}{100}w_1^2,
\end{align}
and by
 the choice of $\mu_0$ and  \eqref{5.16a}, we have
\begin{align}
\frac{99}{100}
\le 1-G_u \le\frac {101}{100}.
\label{wi2aa}
\end{align}
From the above choices, we shall separate two steps to derive the  estimate of $|Du|(x_0,t_0)$. As we mentioned before, all the calculations will be done at the fixed point $(x_0,t_0)$.

{\bf Step 1 :} We first get the formula \eqref{3I2c}.

Taking the first  derivatives of $\varphi$,
\begin{align}\label{3varphit}
\varphi_t=&\frac{(|Dw|^2)_t}{|Dw|^2\log|Dw|^2}+h'u_t,
\end{align}
\begin{align}\label{3varphii}
\varphi_i=&\frac{(|Dw|^2)_i}{|Dw|^2\log|Dw|^2}+h'u_i+g'\gamma^i.
\end{align}
From $\varphi_i(x_0,t_0)=0$,  we have
\begin{align}\label{3varphii=0}
(|Dw|^2)_i=-|Dw|^2\log|Dw|^2(h'u_i+g'\gamma^i)=2w_1^2\log w_1 (h'u_i+g'\gamma^i).
\end{align}
Take the derivatives again for  $\varphi_i$,
\begin{align}\label{3varphiija}
\varphi_{ij}
=&\frac{(|Dw|^2)_{ij}}{|Dw|^2\log|Dw|^2}-(1+\log|Dw|^2)\frac{(|Dw|^2)_i(|Dw|^2)_j}{(|Dw|^2\log|Dw|^2)^2}\notag\\&
+h'u_{ij}
+h''u_iu_j+g''\gamma^i\gamma^j+g'(\gamma^i)_j.
\end{align}
Using \eqref{3varphii=0}, it follows that
\begin{align}\label{3varphiijb}
\varphi_{ij}
=&\frac{(|Dw|^2)_{ij}}{|Dw|^2\log|Dw|^2}+h'u_{ij}-(1+\log|Dw|^2)h'^2u_iu_j\notag\\&
-(1+\log|Dw|^2)g'^2\gamma^i\gamma^j
-(1+\log|Dw|^2)h'g'(\gamma^iu_j+\gamma^ju_i)+g'(\gamma^i)_j.
\end{align}
Then we get
\begin{align}\label{3aijvarphiija}
0\geq \sum_{1\leq i,j\leq n}a^{ij}\varphi_{ij}-\varphi_t
=:&I_1+I_2,
\end{align}
where
\begin{align}\label{3I1a}
I_1=\frac{1}{|Dw|^2\log|Dw|^2}\big[\sum_{1\leq i,j\leq n}a^{ij}(|Dw|^2)_{ij}-(|Dw|^2)_t\big],
\end{align}
and
\begin{align}
I_2=&\sum_{1\leq i,j\leq n}a^{ij}\bigg\{h' u_{ij}-(1+\log|Dw|^2)h'^2u_iu_j-(1+\log|Dw|^2)g'^2\gamma^i\gamma^j\notag\\
&\qquad\qquad\quad-2(1+\log|Dw|^2)h'g'\gamma^iu_j+g'(\gamma^i)_j\bigg\}-h'u_t.\label{3I2a}
\end{align}

Now we first treat $I_2$.
From the choice of the coordinate and the equations \eqref{1.1}, \eqref{3.1}, we have
\begin{align}
I_2
\ge&f-2(1+\alpha_0^2\sum_{1\leq i,j\leq n}a^{ij}\gamma^i\gamma^j)\log w_1-C_4.\label{3I2c}
\end{align}

{\bf Step 2 :}
 We calculate $I_1$ and get the formula \eqref{3I1c}. Then we finish the proof of gradient estimate \eqref{1.12}.

Taking the first  derivatives of $|Dw|^2$, we have
\begin{align}\label{3|Dw|^2t}
(|Dw|^2)_t=&2w_1w_{1t},
\end{align}
\begin{align}\label{3|Dw|^2i}
(|Dw|^2)_i=&2w_1w_{1i}.
\end{align}
Taking the derivatives of $(|Dw|^2)_i$ with respect to $x_j$, we have
\begin{align}\label{3|Dw|^2ij}
(|Dw|^2)_{ij}=&2w_1w_{1ij}+2w_{ki}w_{kj}.
\end{align}
By \eqref{3I1a} and \eqref{3|Dw|^2ij}, we can rewrite $I_1$ as
\begin{align}
I_1
=&\frac{1}{w_1\log w_1}\big[\sum_{1\leq i,j\leq n}a^{ij}w_{ij1}-w_{1t}\big]+\frac{1}{w_1^2\log w_1}\sum_{1\leq i,j,k\leq n}a^{ij}w_{ki}w_{kj}.\label{3I1b}
\end{align}
In the following, we shall deal with $I_{1}$. As we have let
\begin{align}
w=&u-G,\quad G=\psi(x,u)d,\label{3w}
\end{align}
then we have
\begin{align}
w_1=&(1-G_u)u_1-G_{x_1},\notag\\
w_{1i}=&(1-G_u)u_{1i}-G_{uu}u_1u_i-G_{ux_i}u_1-G_{ux_1}u_i-G_{x_1x_i},\label{3wki}\\
w_{1ij}=&(1-G_u)u_{1ij}-G_{uu}(u_{1i}u_j+u_{1j}u_i+u_{ij}u_1)\notag\\
&-G_{ux_i}u_{1j}-G_{ux_j}u_{1i}-G_{ux_1}u_{ij}\notag\\
&-G_{uuu}u_1u_iu_j-G_{uux_i}u_ju_1-G_{uux_j}u_iu_1-G_{uux_1}u_iu_j\notag\\
&-G_{ux_ix_j}u_1-G_{ux_1x_j}u_i-G_{ux_ix_1}u_j-G_{x_ix_jx_1}.\label{3wkij}
\end{align}
So from the choice of the coordinate and the equation \eqref{1.1},  we have
\begin{align}
\sum_{1\leq i,j\leq n}a^{ij}w_{ij1}-w_{1t}
\geq&(1-G_u)[\sum_{1\leq i,j\leq n}a^{ij}u_{ij1}-u_{1t}]-(G_{uu}u_1+G_{ux_1})f\notag\\&
-2G_{uu}\sum_{1\leq i,j\leq n}a^{ij}u_{1i}u_j-2\sum_{1\leq i,j\leq n}a^{ij}G_{ux_i}u_{1j}-C_5u_1.\label{3wkaijwkij}
\end{align}
Differentiating \eqref{1.1}, we have
\begin{align}\label{3aijuijka}
\sum_{1\leq i,j\leq n}a^{ij}u_{ij1}-u_{t1}=&-\sum_{1\leq i,j,l\leq n}a^{ij}_{p_l}u_{l1}u_{ij}+D_{1}f.
\end{align}
and
\begin{align}\label{3aijpl}
a^{ij}_{p_l}=&\frac{2u_iu_ju_l}{v^4}-\frac{\delta_{il}u_j+\delta_{jl}u_i}{v^2}.
\end{align}
From \eqref{3aijpl}, we have
\begin{align}\label{3aijuijkb}
\sum_{1\leq i,j\leq n}a^{ij}u_{ij1}-u_{t1}
=&\frac{2}{v^2}\sum_{1\leq i,j,l\leq n}a^{il}u_{l1}u_{ij}u_{j}+D_{1}f.
\end{align}
Inserting \eqref{3aijuijkb} into \eqref{3wkaijwkij} and \eqref{3I1b}, we have
\begin{align}
I_1
\geq&\frac{1}{w_1^2\log w_1}\big[\sum_{1\leq i,j,k\leq n}a^{ij}w_{ki}w_{kj}+\frac{2(1-G_u)}{v^2}w_1\sum_{1\leq i,j,l\leq n}a^{il}u_{l1}u_{ij}u_{j}\notag\\&
-2w_1\sum_{1\leq i,j\leq n}a^{ij}(G_{uu}u_{i}+G_{ux_i})u_{1j}+(1-G_u)w_1D_{1}f\notag\\&
-(G_{uu}u_1+G_{ux_1})w_1f-C_7u_1^2\big].\label{3I1c}
\end{align}
Next we shall treat the second derivative terms in \eqref{3I1c}, using the first order derivative condition
 $\varphi_i(x_0,t_0)=0.$
 By \eqref{3varphii=0}  and  \eqref{3|Dw|^2i},  we have
\begin{align}\label{3wkwki=1}
w_{1i}
=&-w_1\log w_1(h'u_i+g'\gamma^i),\qquad i=1,2,\ldots,n.
\end{align}
Putting \eqref{3wki} into \eqref{3wkwki=1},  by the choice of coordinate, we get
\begin{align}\label{3wkuki=1}
(1-G_u)u_{1i}
=&-w_1\log w_1(h'u_i+g'\gamma^i)+(G_{uu}u_1+ G_{ux_1})u_i+ (u_1G_{ux_i}+ G_{x_1x_i}),\\
&\hspace{200pt} i=1,2,\ldots,n.\notag
\end{align}
It follows that
\begin{align}\label{uij}
 &\sum_{1\leq i,j,k\leq n}a^{ij}w_{ki}w_{kj}+\frac{2(1-G_u)}{v^2}w_1\sum_{1\leq i,j,l\leq n}a^{il}u_{l1}u_{ij}u_{j}
-2w_1\sum_{1\leq i,j\leq n}a^{ij}(G_{uu}u_{i}+G_{ux_i})u_{1j}\notag\\
&\geq3(1+\alpha_0^2\sum_{1\leq i,j\leq n}a^{ij}\gamma^i\gamma^j)w_1^2\log^2 w_1-C_8w_1^2\log w_1,
\end{align}
and
\begin{align}\label{3Dkf}
(1-G_u)D_1f=&(1-G_u)f_{x_1}+(1-G_u)f_uu_1+(1-G_u)\sum_{1\leq j\leq n}f_{p_j}u_{j1}\notag\\
=&(1-G_u)f_{x_1}+(1-G_u)f_uu_1-h'w_1\log w_1\sum_{1\leq j\leq n}f_{p_j}u_j\notag\\&+(G_{uu}u_1+ G_{ux_1})\sum_{1\leq j\leq n}f_{p_j}u_j
-g'w_1\log w_1\sum_{1\leq j\leq n}f_{p_j}\gamma^j\notag\\&+ \sum_{1\leq j\leq n}f_{p_j}(u_1G_{ux_j}+ G_{x_1x_j}),
\end{align}
Inserting \eqref{uij} and \eqref{3Dkf} into \eqref{3I1c},  and combining \eqref{3I1c}, we get
\begin{align}\label{3aijvarphiijc}
0\geq& \sum_{1\leq i,j\leq n}a^{ij}\varphi_{ij}-\varphi_t
\geq\frac{1}{4}\log w_1-C_9.
\end{align}
Since $\varphi(x,t)\le \varphi(x_0,t_0)$,  we get the gradient estimate \eqref{1.12} and thus complete the proof of Theorem~\ref{Thm1.1}.\qed

Consequently, equation \eqref{1.1} is uniformly parabolic in $\Omega\times [0,T)$ for any fixed $T>0$. By the standard theory,
we obtain the long time existence of solutions. That is Theorem~\ref{Thm1.2}.

Thus Huisken's result in \cite{Huis89} can be as an application. We have the following corollary.
\begin{Cor}\label{Cor1.1}
Assume that $\Omega \subset \mathbb R^n $ is a bounded  $C^{2,\alpha}$ domain for some $\alpha\in (0,1)$, $n\geq 2$. Let  $\gamma$ be the inward unit normal to $\partial\Omega $.
Problem  \eqref{1.1}-\eqref{1.3} for $f=0$ and $\psi=0$ admits a   solution $ u\in C^{2,\alpha}(\overline\Omega\times[0,\infty))$
satisfying the estimates \eqref{1.11a},\eqref{1.11},\eqref{1.12} and asymptotically converges to a constant function as $t\rightarrow\infty$.
\end{Cor}

  {\bf Acknowledgment.}
The author would like to
thank Professor Xinan Ma for  helpful discussions and encouragement.

\end{document}